\documentclass[11pt,a4paper]{article}

\usepackage{inputenc}
\usepackage{amsmath}
\usepackage{bm}
\usepackage{bbold}
\usepackage{amsthm}

\usepackage{hyperref}

\title{Bounds on Mean Cycle Time\\ in Acyclic Fork-Join Queueing Networks\thanks
{Proceedings of the 4th International Workshop on Discrete Event Systems (WODES'98), University of Cagliari, Cagliari, Sardinia, Italy, August 26-28, 1998, London, IEE, 1998, pp.~469--474.}} 

\author{Nikolai K. Krivulin\thanks{Faculty of Mathematics and Mechanics, St.~Petersburg State University, 28 Universitetsky Ave., St.~Petersburg, 198504, Russia, 
nkk@math.spbu.ru}
}

\date{}

\newtheorem{theorem}{Theorem}
\newtheorem{lemma}[theorem]{Lemma}

\newtheorem{proposition}{Proposition}

\def\sumo_#1^#2{\setbox0=\hbox{$\displaystyle{\sum}$}
                \setbox1=\hbox{$\scriptstyle{#1}$}
                \setbox2=\hbox{$\scriptstyle{#2}$}
		\setbox3=\hbox{${}_{{}_\oplus}\mathsurround=0pt$}
		\dimen1=.5\wd1 \advance\dimen1 by-.5\wd0
		\ifdim\dimen1>0pt
		   \ifdim\dimen1>\wd3 \kern\wd3 \else\kern\dimen1\fi\fi
		\dimen2=.5\wd2 \advance\dimen2 by-.5\wd0
		\ifdim\dimen2>0pt
		   \ifdim\dimen2>\wd3 \kern\wd3 \else\kern\dimen2\fi\fi
		\mathop{{\sum}{}_{{}_\oplus}}_{\kern-\wd3 #1}^{\kern-\wd3 #2}}
\def\sumol_#1{\setbox0=\hbox{$\displaystyle{\sum}$}
             \setbox1=\hbox{$\scriptstyle{#1}$}
	     \setbox3=\hbox{${}_{{}_\oplus}\mathsurround=0pt$}
	     \dimen1=.5\wd1 \advance\dimen1 by-.5\wd0
	     \ifdim\dimen1>0pt
	        \ifdim\dimen1>\wd3 \kern\wd3 \else\kern\dimen1\fi\fi
	     \mathop{{\sum}{}_{{}_\oplus}}_{\kern-\wd3 #1}}
\def\prodo_#1^#2{\setbox0=\hbox{$\displaystyle{\prod}$}
                \setbox1=\hbox{$\scriptstyle{#1}$}
                \setbox2=\hbox{$\scriptstyle{#2}$}
		\setbox3=\hbox{${}_{{}_\otimes}\mathsurround=0pt$}
		\dimen1=.5\wd1 \advance\dimen1 by-.5\wd0
		\ifdim\dimen1>0pt
		   \ifdim\dimen1>\wd3 \kern\wd3 \else\kern\dimen1\fi\fi
		\dimen2=.5\wd2 \advance\dimen2 by-.5\wd0
		\ifdim\dimen2>0pt
		   \ifdim\dimen2>\wd3 \kern\wd3 \else\kern\dimen2\fi\fi
	      \mathop{{\prod}{}_{{}_\otimes}}_{\kern-\wd3 #1}^{\kern-\wd3 #2}}
\def\fraco#1#2{\setbox1=\hbox{$\displaystyle{#1}$}
               \setbox2=\hbox{$\displaystyle{#2}$}
               \dimen0=\wd1
               \ifdim\wd2>\dimen0 \dimen0=\wd2\fi
               \advance \dimen0 by .4em
               \setbox3=\vbox{\hrule width\dimen0
                  \vskip1.3pt \hrule width\dimen0}
             \mathop{\raise1.6pt \box3}_{\raise1.6pt\box2}^{\raise1.6pt\box1}}

\newcommand{\normo}[1]{\left\|#1\right\|_{\scriptscriptstyle\oplus}}

\begin{document}

\maketitle

\begin{abstract}
Simple lower and upper bounds on mean cycle time in stochastic acyclic
fork-join networks are derived using the $(\max,+)$-algebra approach. The
behaviour of the bounds under various assumptions concerning the service times
in the networks is discussed, and related numerical examples are presented.
\\

\textit{Key-Words:} max-plus algebra, dynamic state equation, acyclic fork-join queueing networks,
Mean cycle time.
\end{abstract}

\section{Introduction}
One of the problems of interest in the analysis of stochastic queueing
networks is to evaluate the mean cycle time of a network. Both the mean cycle
time and its inverse which can be regarded as a throughput, present
performance measures commonly used to describe efficiency of the network
operation.

It is frequently rather difficult to evaluate the mean cycle time exactly,
even though the network under study is quite simple. To get information about
the performance measure in this case, one can apply computer simulation to
produce reasonable estimates. Another approach is to derive bounds on the mean
cycle time (see examples in \cite{Bacc91,Bacc93,Glas95}).

The paper is concerned with the derivation of a lower and upper bounds on the
mean cycle time for stochastic acyclic fork-join queueing networks
\cite{Bacc89,Bacc93}. A useful way to represent dynamics of the networks is
based on the $(\max,+)$-algebra approach \cite{Cuni79,Bacc93,Masl94}. We
apply the $(\max,+)$-algebra dynamic representation proposed in
\cite{Kriv96a,Kriv96b} to get algebraic bounds and then exploit them to derive
bounds in the stochastic case.

Compared to the techniques proposed in \cite{Bacc91,Glas95}, which mainly rely
on results of the theory of large deviations as well as the Perron-Frobenius
spectral theory, our approach is essentially based on pure algebraic
manipulations combined with application of bounds on extreme values, obtained
in \cite{Gumb54,Hart54}.

\section{Algebraic Definitions and Results}\label{S-ADR}
The $(\max,+)$-algebra is an idempotent commutative semiring (idempotent
semifield) which is defined as
$ {\mathbb R}_{\rm max}=(\underline{{\mathbb R}},\oplus,\otimes) $ with
$ \underline{{\mathbb R}} = {\mathbb R} \cup \{\varepsilon\} $,
$ \varepsilon = -\infty $, and binary operations $ \oplus $ and
$ \otimes $ defined as
$$
x \oplus y = \max(x,y), \quad x \otimes y = x + y \quad
\forall x,y\in\underline{\mathbb R}.
$$

As it is easy to see, the operations $ \oplus $ and $ \otimes $
retain most of the properties of the ordinary addition and multiplication,
including associativity, commutativity, and distributivity of multiplication
over addition. However, the operation $ \oplus $ is idempotent; that
is, for any $ x \in \underline{{\mathbb R}} $, one has
$ x \oplus x = x $.

There are the null and identity elements, namely $ \varepsilon $ and
$ 0 $, to satisfy the conditions
$ x \oplus \varepsilon = \varepsilon \oplus x = x $, and
$ x \otimes 0 = 0 \otimes x = x $, for any
$ x \in \underline{{\mathbb R}} $. The null element $ \varepsilon $ and
the operation $ \otimes $ are related by the usual absorption rule
involving $ x \otimes \varepsilon = \varepsilon \otimes x = \varepsilon $.

Non-negative power of any $ x\in{\mathbb R} $ is defined as
$$
x^{\otimes 0}=0, \quad
x^{\otimes q}
=\underbrace{x\otimes\cdots\otimes x}_{q\;\;{\rm times}}
$$
for any integer $ q\geq1 $. Clearly, the $(\max,+)$-algebra power
$ x^{\otimes q} $ corresponds to $ qx $ in ordinary notations.

\subsection{Algebra of Matrices}
The $(\max,+)$-algebra of matrices is readily introduced in the regular way.
Specifically, for any $(n \times n)$-matrices $ X = (x_{ij}) $ and
$ Y = (y_{ij}) $, the entries of $ U = X \oplus Y $ and
$ V = X \otimes Y $ are calculated as
$$
u_{ij} = x_{ij} \oplus y_{ij}, \quad \mbox{and} \quad
v_{ij} = \sumo_{k=1}^{n} x_{ik} \otimes y_{kj},
$$
where $ \sum_{\scriptscriptstyle\oplus} $ stands for the iterated
operation $ \oplus $. As the null element, the matrix
$ \mathcal{E} $ with all entries equal to $ \varepsilon $ is taken in the
algebra, whereas the diagonal matrix $ E = \mathop\mathrm{diag}(0,\ldots,0) $ with the
off-diagonal entries set to $ \varepsilon $ presents the identity.

For any square matrix $ X\ne\mathcal{E} $, one can define
$$
X^{\otimes 0} = E, \quad
X^{\otimes q}
=\underbrace{X\otimes\cdots\otimes X}_{q\;\;{\rm times}}
$$
for any integer $ q\geq 1 $. However, idempotency in this algebra leads, in
particular, to the matrix identity
$$
(E \oplus X)^{\otimes q} = E \oplus X \oplus \cdots \oplus X^{\otimes q}.
$$

For any matrix $ X $, its norm is defined as
$$
\normo{X}=\sumol_{i,j}x_{ij}=\max_{i,j}x_{ij}.
$$
The matrix norm possesses the usual properties. Specifically, for any matrix
$ X $, it holds $ \normo{X}\geq\varepsilon $, and
$ \normo{X}=\varepsilon $ if and only if $ X=\mathcal{E} $. Furthermore,
$ \normo{c\otimes X}=c\otimes\normo{X} $ for any
$ c\in\underline{\mathbb R} $, and
\begin{eqnarray*}
\normo{X\oplus Y}  & =  &\normo{X}\oplus\normo{Y}, \\
\normo{X\otimes Y} &\leq& \normo{X}\otimes\normo{Y}
\end{eqnarray*}
for any two conforming matrices $ X $ and $ Y $. Note that for any
$ c>0 $, we also have $ \normo{cX}=c\normo{X} $.

Consider an $(n \times n)$-matrix $ X $ with its entries
$ x_{ij} \in \underline{{\mathbb R}} $. It can be treated as an adjacency
matrix of an oriented graph with $ n $ vertices, provided each entry
$ x_{ij} \neq \varepsilon $ implies the existence of the edge
$ (i,j) $ in the graph, while $ x_{ij} = \varepsilon $ does the lack
of the edge.

It is easy to verify that for any integer $ q\geq 1 $, the matrix
$ X^{\otimes q} $ has its the entry $ x^{(q)}_{ij}\ne\varepsilon $ if
and only if there exists a path from vertex $ i $ to vertex $ j $ in
the graph, which consists of $ q $ edges. Furthermore, if the graph
associated with the matrix $ X $ is acyclic, we have
$ X^{\otimes q}=\mathcal{E} $ for all $ q>p $, where $ p $ is the
length of the longest path in the graph. Otherwise, provided that the graph is
not acyclic, one can construct a path of any length, lying along circuits, and
then it holds that $ X^{\otimes q}\neq\mathcal{E} $ for all $ q \geq 0 $.

\subsection{Further Algebraic Results}
For any graph, its adjacency matrix $ G $ with the elements equal either
to $ 0 $ or $ \varepsilon $ is said to be standard. It is easy to
verify the next statement.
\begin{proposition}\label{P-XXG}
For any matrix $ X $, it holds
$$
X \le \normo{X}\otimes G,
$$
where $ G $ is the standard adjacency matrix of the graph associated with
$ X $.
\end{proposition}

In particular, for a matrix $ D=\mathop\mathrm{diag}(d_{1},\ldots,d_{n}) $, we have
$$
D \le \normo{D}\otimes E = \left(\sumo_{i=1}^{n}d_{i}\right) \otimes E.
$$

\begin{proposition}\label{P-AAE}
Let matrices $ X_{1},\ldots,X_{k} $ have a common associated acyclic
graph, $ p $ be the length of the longest path in the graph, and
$$
X = X_{1}^{\otimes m_{1}}\otimes\cdots\otimes X_{k}^{\otimes m_{k}},
$$
where $ m_{1},\ldots,m_{k} $ are nonnegative integers.

If it holds that $ m_{1}+\cdots+m_{k}>p $, then $ X=\mathcal{E} $.
\end{proposition}
\begin{proof}
$$
X
=
X_{1}^{\otimes m_{1}}\otimes\cdots\otimes X_{k}^{\otimes m_{k}} \\
\le
\normo{X_{1}}^{\otimes m_{1}}\otimes\cdots\otimes
              \normo{X_{k}}^{\otimes m_{k}}\otimes G^{\otimes m},
$$
where $ m=m_{1}+\cdots+m_{k} $.

Since the associated graph is acyclic, it holds that
$ G^{\otimes q}=\mathcal{E} $ for all $ q>p $. Therefore,
if $ m>p $, then $ G^{\otimes m}=\mathcal{E} $, and we arrive at the
inequality $ X \le \mathcal{E} $ which leads us to the desired result.
\end{proof}

\begin{lemma}\label{L-EAA}
Let matrices $ X_{1},\ldots,X_{k} $ have a common associated acyclic
graph, and $ p $ be the length of the longest path in the graph.

If $ \normo{X_{i}}\ge 0 $ for all $ i=1,\ldots,k $, then for any
nonnegative integers $ m_{1},\ldots,m_{k} $, it holds
$$
\normo{\prodo_{i=1}^{k}(E\oplus X_{i})^{\otimes m_{i}}}
\leq
\left(\sumo_{i=1}^{k}\normo{X_{i}}\right)^{\otimes p},
$$
where $ \prod_{\scriptscriptstyle\otimes} $ denotes the iterated operation
$ \otimes $.
\end{lemma}

\begin{proof}
Consider the matrix
$$
X = \prodo_{i=1}^{k}(E\oplus X_{i})^{\otimes m_{i}},
$$
and represent it in the form
\begin{eqnarray*}
X & = &
(E\oplus X_{1})^{\otimes m_{1}}
\otimes\cdots\otimes
(E\oplus X_{k})^{\otimes m_{k}} \\
  & = &
\sumo_{i_{1}=0}^{m_{1}}X_{1}^{\otimes i_{1}}\otimes\cdots\otimes
\sumo_{i_{k}=0}^{m_{k}}X_{k}^{\otimes i_{k}} \\
  & = &
\sumo_{i_{1}=0}^{m_{1}}\cdots\sumo_{i_{k}=0}^{m_{k}}
X_{1}^{\otimes i_{1}}\otimes\cdots\otimes X_{k}^{\otimes i_{k}} \\
  & \le &
\sumol_{0\leq i_{1}+\cdots+i_{k}\leq m}
X_{1}^{\otimes i_{1}}\otimes\cdots\otimes X_{k}^{\otimes i_{k}},
\end{eqnarray*}
where $ m=m_{1}+\cdots+m_{k} $. From Proposition~\ref{P-AAE} we may replace
$ m $ with $ p $ in the last term to get
\begin{eqnarray*}
X & \le & \sumol_{0\leq i_{1}+\cdots+i_{k}\leq p}
X_{1}^{\otimes i_{1}}\otimes\cdots\otimes X_{k}^{\otimes i_{k}}.
\end{eqnarray*}

Proceeding to the norm, with its additive and multiplicative properties, we
arrive at the inequality
$$
\normo{X}
\leq
\sumol_{0\leq i_{1}+\cdots+i_{k}\leq p}
\normo{X_{1}}^{\otimes i_{1}}\otimes\cdots\otimes
\normo{X_{k}}^{\otimes i_{k}}.
$$
Since
$ 0\leq\normo{X_{i}}\leq\normo{X_{1}}\oplus\cdots\oplus\normo{X_{k}} $
for all $ i=1,\ldots,k $, we finally have
\begin{eqnarray*}
\hskip 8.5mm
\normo{X}
& \leq &
\sumo_{i=0}^{p}(\normo{X_{1}}\oplus\cdots\oplus\normo{X_{k}})^{\otimes p} \\
& = &
(\normo{X_{1}}\oplus\cdots\oplus\normo{X_{k}})^{\otimes p}.
\qedhere
\end{eqnarray*}
\end{proof}

\section{Elements of Probability}\label{S-EOP}
In this section we present some probabilistic results associated with
$(\max,+)$-algebra concepts introduced above. We start with a lemma
which states properties of the expectation with respect to the operations
$ \oplus $ and $ \otimes $.

\begin{lemma}\label{L-POE}
Let $ \xi_{1},\ldots,\xi_{k} $ be random variables taking their values in
$ \underline{\mathbb R} $, and such that their expected values
$ {\mathbb E}[\xi_{i}] $, $ i=1,\ldots,k $, exist. Then it holds
\begin{enumerate}
\item $ {\mathbb E}[\xi_{1}\oplus\cdots\oplus\xi_{k}]
\geq{\mathbb E}[\xi_{1}]\oplus\cdots\oplus{\mathbb E}[\xi_{k}] $,
\item $ {\mathbb E}[\xi_{1}\otimes\cdots\otimes\xi_{k}]
=   {\mathbb E}[\xi_{1}]\otimes\cdots\otimes{\mathbb E}[\xi_{k}] $.
\end{enumerate}
\end{lemma}
\begin{proof}
Clearly, it will suffice to prove the lemma only for $ k=2 $, and then
extend the proof by induction to the case of arbitrary $ k $.

To verify the first inequality, first assume that one of the expectation on
its right side, say $ {\mathbb E}[\xi_{2}] $, is infinite; that is
$ {\mathbb E}[\xi_{2}]=\varepsilon $. Since it holds that
$ \xi_{1}\oplus\xi_{2}\geq\xi_{1} $, we have
$$
{\mathbb E}[\xi_{1}\oplus\xi_{2}]
\geq {\mathbb E}[\xi_{1}]
=    {\mathbb E}[\xi_{1}]\oplus{\mathbb E}[\xi_{2}].
$$

Suppose now that
$ {\mathbb E}[\xi_{1}], {\mathbb E}[\xi_{2}] > \varepsilon $, and consider
the obvious identity
$$
x\oplus y = \frac{1}{2}(x+y+|x-y|) \quad \forall x,y \in {\mathbb R}.
$$
With ordinary properties of expectation, we get
$$
{\mathbb E}[\xi_{1}\oplus\xi_{2}]
\geq
\frac{1}{2}({\mathbb E}[\xi_{1}]+{\mathbb E}[\xi_{2}]
 + |{\mathbb E}[\xi_{1}]-{\mathbb E}[\xi_{2}]|) \\
=
{\mathbb E}[\xi_{1}]\oplus{\mathbb E}[\xi_{2}].
$$

The second assertion of the lemma is trivial.
\end{proof}

The next result \cite{Gumb54,Hart54} provides an upper bound for the expected
value of the maximum of independent and identically distributed (i.i.d.)
random variables.
\begin{lemma}\label{L-EED}
Let $ \xi_{1},\ldots,\xi_{k} $ be i.i.d. random variables with
$ {\mathbb E}[\xi_{1}]<\infty $ and
$ {\mathbb D}[\xi_{1}]<\infty $. Then it holds
$$
{\mathbb E}\left[\sumo_{i=1}^{k}\xi_{i}\right]
\le
{\mathbb E}[\xi_{1}]+\frac{k-1}{\sqrt{2k-1}}\sqrt{{\mathbb D}[\xi_{1}]}.
$$
\end{lemma}

Consider a random matrix $ X $ with its entries $ x_{ij} $ taking
values in $ \underline{\mathbb R} $. We denote by $ {\mathbb E}[X] $ the
matrix obtained from $ X $ by replacing each entry $ x_{ij} $ by its
expected value $ {\mathbb E}[x_{ij}] $.

\begin{lemma}\label{L-EXE}
For any random matrix $ X $, it holds
$$
{\mathbb E}\normo{X}\geq\normo{{\mathbb E}[X]}.
$$
\end{lemma}

\begin{proof}
It follows from Lemma~\ref{L-POE} that
$$
{\mathbb E}\normo{X}
=
{\mathbb E}\left[\sumol_{i,j}x_{ij}\right]
\geq
\sumol_{i,j}{\mathbb E}[x_{ij}]
=
\normo{{\mathbb E}[X]}.
\qedhere
$$
\end{proof}

\section{Acyclic Fork-Join Networks}\label{S-AFJ}
We consider a network with $ n $ single-server nodes and customers of a
single class. The network topology is described by an oriented acyclic graph
with the set of vertices which represent the nodes of the network, and the set
of edges which determine the transition routes of customers. It is assumed
that there are vertices in the graph which have no either incoming or outgoing
edges. Each vertex with no predecessors is assumed to represent an infinite
external arrival stream of customers; provided that a vertex has no successors,
it is considered as an output node intended to release customers from the
network.

Each node of the network includes a server and infinite buffer which
together present a single-server queue operating under the first-come,
first-served queueing discipline. It is assumed that at the initial time, all
the servers are free of customers, the buffers in the nodes with no
predecessors have infinite number of customers, whereas the buffers in the
other nodes are empty.

In addition to the usual service procedure, some special join and fork
operations \cite{Bacc89} may be performed in a node respectively before and
after service of a customer. The join operation is actually thought to cause
each customer which comes into a node not to enter the buffer at the server
but to wait until at least one customer from every preceding node arrives. As
soon as these customers arrive, they, taken one from each preceding node, are
united to be treated as being one customer which then enters the buffer to
become a new member of the queue.

The fork operation at a node is initiated every time the service of a
customer is completed; it consists in giving rise to several new customers
one for each succeeding nodes. These customers simultaneously depart the node,
each being passed to separate node related to the first one. We assume that
the execution of fork-join operations when appropriate customers are
available, as well as the transition of customers within and between nodes
require no time.

An example of an acyclic fork-join network with $ n=5 $ is shown in
Fig.~\ref{F-AFJ}.
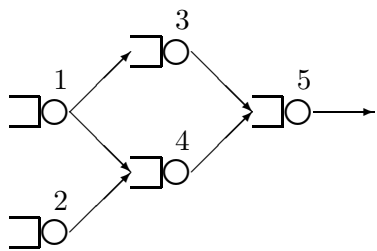
\begin{figure}[hhh]
\begin{center}
\setlength{\unitlength}{1mm}
\begin{picture}(50,35)
\newsavebox\queue
\savebox{\queue}(10,6){\thicklines
 \put(0,2){\line(1,0){4}}
 \put(0,-2){\line(1,0){4}}
 \put(4,2){\line(0,-1){4}}
 \put(6,0){\circle{3}}}

\put(0,16){\usebox\queue}
\put(5.5,22){$1$}
\put(8,19){\vector(1,1){8}}
\put(8,19){\vector(1,-1){8}}
\put(0,0){\usebox\queue}
\put(5.5,6){$2$}
\put(8,3){\vector(1,1){8}}

\put(16,24){\usebox\queue}
\put(21.5,30){$3$}
\put(24,27){\vector(1,-1){8}}
\put(16,8){\usebox\queue}
\put(21.5,14){$4$}
\put(24,11){\vector(1,1){8}}

\put(32,16){\usebox\queue}
\put(37.5,22){$5$}
\put(40,19){\vector(1,0){8}}

\end{picture}
\end{center}
\caption{An acyclic fork-join network.}\label{F-AFJ}
\end{figure}
Note that open tandem queueing systems (see Fig.~\ref{F-OTQ}) can be
considered as trivial fork-join networks in which no fork and join
operations are actually performed.
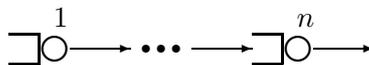
\begin{figure}[hhh]
\begin{center}
\setlength{\unitlength}{1mm}
\begin{picture}(50,10)
\savebox{\queue}(10,6){\thicklines
 \put(0,2){\line(1,0){4}}
 \put(0,-2){\line(1,0){4}}
 \put(4,2){\line(0,-1){4}}
 \put(6,0){\circle{3}}}

\put(0,0){\usebox\queue}
\put(5.5,6){$1$}
\put(8,3){\vector(1,0){8}}
\multiput(18,3)(2,0){3}{\circle*{1}}
\put(24,3){\vector(1,0){8}}
\put(32,0){\usebox\queue}
\put(37.5,6){$n$}
\put(40,3){\vector(1,0){8}}

\end{picture}
\end{center}
\caption{Open tandem queues.}\label{F-OTQ}
\end{figure}

For the queue at node $ i $, we denote the $k$th departure epoch by
$ x_{i}(k) $, and the service time of the $k$th customer by $ \tau_{ik} $.
We assume that $ \tau_{ik}\geq0 $ are given parameters for all
$ i=1,\ldots,n $, and $ k=1,2,\ldots $, while $ x_{i}(k) $ are
considered as unknown state variables. With the condition that the network
starts operating at time zero, it is convenient to set
$ x_{i}(0) = 0 $ for all $ i=1,\ldots,n $.

In order to describe the network dynamics, we apply the $(\max,+)$-algebra
representation derived in \cite{Kriv96a,Kriv96b}. For the network model under 
study, the representation takes the form of the state equation
\begin{equation}\label{E-DSE}
\mbox{\boldmath $x$}(k)=A(k)\otimes\mbox{\boldmath $x$}(k-1),
\end{equation}
with $ \mbox{\boldmath $x$}(k)=(x_{1}(k),\ldots,x_{n}(k))^{T} $ and
\begin{equation}\label{E-STM}
A(k) = (E\oplus\mathcal{T}_{k}\otimes G^{T})^{\otimes p}\otimes\mathcal{T}_{k},
\end{equation}
where $ \mathcal{T}_{k}=\mathop\mathrm{diag}(\tau_{1k},\ldots,\tau_{nk}) $, the matrix
$ G $ is the standard adjacency matrix of the network graph, and $ p $
is the length of the longest path in the graph.

We consider the evolution of the system as a sequence of service cycles:
the 1st cycle starts at the initial time, and it is terminated as soon as all
the servers in the network complete their 1st service, the 2nd cycle is
terminated as soon as the servers complete their 2nd service, and so on.
Clearly, the completion time of the $k$th cycle can be represented as
$$
\max_{i}x_{i}(k)=\normo{\mbox{\boldmath $x$}(k)}.
$$
With the condition $ \mbox{\boldmath $x$}(0)=\mbox{\boldmath $0$} $, we
have from (\ref{E-DSE})
$$
\normo{\mbox{\boldmath $x$}(k)}
=
\normo{A(k)\otimes\cdots\otimes A(1)}.
$$

The next lemma provides simple lower and upper bounds for the norm of the
matrix
$$
A_{k} = A(k)\otimes\cdots\otimes A(1).
$$
\begin{lemma}\label{L-TTT}
For all $ k=1,2,\ldots $, it holds
$$
\normo{\sum_{i=1}^{k}\mathcal{T}_{i}}
\le
\normo{A_{k}}
\le
\sum_{i=1}^{k}\normo{\mathcal{T}_{i}}
+ p\left(\sumo_{i=1}^{k}\normo{\mathcal{T}_{i}}\right).
$$
\end{lemma}
\begin{proof}
To prove the left inequality note that
$$
\left(E\oplus\mathcal{T}_{i}\otimes G^{T}\right)^{\otimes p}\otimes\mathcal{T}_{i}
\geq
\mathcal{T}_{i}
$$
for each $ i=1,\ldots,k $, and so
$ A_{k}\geq\mathcal{T}_{k}\otimes\cdots\otimes\mathcal{T}_{1} $. Proceeding to the
norm, and considering that $ \mathcal{T}_{i} $, $ i=1,\ldots,k $, present
diagonal matrices, we get
\begin{eqnarray*}
\normo{A_{k}}
& \ge &
\normo{\mathcal{T}_{k}\otimes\cdots\otimes\mathcal{T}_{1}} \\
& = &
\normo{\mathcal{T}_{1}+\cdots+\mathcal{T}_{k}}.
\end{eqnarray*}

With Proposition~\ref{P-XXG} we have
$$
A_{k}
\le
\prodo_{i=1}^{k}\normo{\mathcal{T}_{i}}
\otimes
\prodo_{i=1}^{k}\left(E\oplus\mathcal{T}_{k-i+1}\otimes G^{T}\right)^{\otimes p}.
$$
By applying Lemma~\ref{L-EAA}, we finally obtain
\begin{eqnarray*}
\hskip 5mm
\normo{A_{k}}
& \le &
\prodo_{i=1}^{k}\normo{\mathcal{T}_{i}}\otimes
\left(\sumo_{i=1}^{k}\normo{\mathcal{T}_{i}}\right)^{\otimes p} \\
& = &
\sum_{i=1}^{k}\normo{\mathcal{T}_{i}}
+p\left(\sumo_{i=1}^{k}\normo{\mathcal{T}_{i}}\right).
\qedhere
\end{eqnarray*}
\end{proof}

In many applications, one is normally interested in investigating the
steady-state mean cycle time; that is, the limit of the ratio
$$
\frac{1}{k}\normo{\mbox{\boldmath $x$}(k)}
=\frac{1}{k}\normo{A_{k}}
$$
as $ k $ tends to $ \infty $. We will consider this problem with
relation to stochastic networks in the next section.

\section{Stochastic Networks}\label{S-SNS}
Suppose that for each node $ i=1,\ldots,n $, the service times
$ \tau_{i1},\tau_{i2},\ldots $, form a sequence of i.i.d. non-negative
random variables with $ {\mathbb E}[\tau_{ik}]<\infty $ and
$ {\mathbb D}[\tau_{ik}]<\infty $ for all $ k=1,2,\ldots $. With these
conditions, $ \mathcal{T}_{k} $ are i.i.d. random matrices, whereas
$ \normo{\mathcal{T}_{k}} $ present i.i.d. random variables with
$ {\mathbb E}\normo{\mathcal{T}_{k}}<\infty $ and
$ {\mathbb D}\normo{\mathcal{T}_{k}}<\infty $ for all $ k=1,2,\ldots $.

Furthermore, since the matrix $ A(k) $ depends only on $ \mathcal{T}_{k} $,
the matrices $ A(1),A(2),\ldots $, also present i.i.d. random matrices. It
is easy to verify using (\ref{E-STM}) that
$ 0\leq{\mathbb E}\normo{A(k)}<\infty $ for all $ k=1,2,\ldots $.

In the analysis of the mean cycle time of the system, one first has to
convince himself that the limit
$$
\lim_{k\to\infty}\frac{1}{k}\normo{A_{k}} = \gamma
$$
exists. As a standard tool to verify the existence of the above limit,
the next Subadditive Ergodic Theorem proposed in \cite{King73} is normally
applied. One can find examples of the implementation of the theorem in the
$(\max,+)$-algebra framework in \cite{Cohe88,Bacc93,Glas95}.
\begin{theorem}\label{T-KSE}
Let $ \{\xi_{lk} | \; l,k=1,2,\ldots; \, l<k \} $ be a family of random
variables which satisfy the following properties:

Subadditivity: $ \xi_{lk} \leq \xi_{lm} + \xi_{mk} $ for all
$ l<m<k $;

Stationarity: the joint distributions are the same for both families
$ \{\xi_{l+1 k+1} | \; l<k \} $ and $ \{\xi_{lk} | \; l<k \} $;

Boundedness: for all $ k=1,2,\ldots $, there exists
$ {\mathbb E}[\xi_{0k}] $, and $ {\mathbb E}[\xi_{0k}] \geq -ck $ for
some finite number $ c $.

Then there exists a constant $ \gamma $, such that it holds
\begin{enumerate}
\item ${\displaystyle \lim_{k \rightarrow \infty} \xi_{0k}/k = \gamma}$
\quad with probability 1,
\item ${\displaystyle \lim_{k \rightarrow \infty} {\mathbb E}[\xi_{0k}]/k
= \gamma}$.
\end{enumerate}
\end{theorem}

In order to apply Theorem~\ref{T-KSE} to stochastic system (\ref{E-DSE})
with transition matrix (\ref{E-STM}), one can define the family of random
variables $ \{\xi_{lk} | \; l<k \} $ with
$$
\xi_{lk}=\normo{A(k)\otimes\cdots\otimes A(l)}.
$$

Since $ A(i) $, $ i=1,2,\ldots $, present i.i.d. random matrices, the
family $ \{\xi_{lk}\}_{l<k} $ satisfies the stationarity condition of
Theorem~\ref{T-KSE}. Furthermore, the multiplicative property of the norm
endows the family with subadditivity. The boundedness condition can be readily
verified using properties of the norm together with the condition that
$ 0\leq{\mathbb E}[\tau_{ik}]<\infty $ for all $ i=1,\ldots,n $, and
$ k=1,2,\ldots $.

Now we are in a position to present our main result which offers bounds on the
mean cycle time.
\begin{theorem}
In the stochastic dynamical system (\ref{E-DSE}) the mean cycle time
$ \gamma $ satisfies the double inequality
\begin{equation}\label{I-MCT}
\normo{{\mathbb E}[\mathcal{T}_{1}]} \leq \gamma \leq
{\mathbb E}\normo{\mathcal{T}_{1}}.
\end{equation}
\end{theorem}
\begin{proof}
Since Theorem~\ref{T-KSE} hods true, we may represent the mean cycle time as
$$
\gamma =
\lim_{k\rightarrow\infty}\frac{1}{k}{\mathbb E}\normo{A_{k}}.
$$

Let us first prove the left inequality in (\ref{I-MCT}). From
Lemmas~\ref{L-TTT} and \ref{L-EXE}, we have
$$
\frac{1}{k}{\mathbb E}\normo{A_{k}}
\geq
\frac{1}{k}{\mathbb E}\normo{\sum_{i=1}^{k}\mathcal{T}_{i}} \\
\geq
\normo{\frac{1}{k}\sum_{i=1}^{k}{\mathbb E}[\mathcal{T}_{i}]}
=
\normo{{\mathbb E}[\mathcal{T}_{1}]},
$$
independently of $ k $.

With the upper bound offered by Lemma~\ref{L-TTT}, we get
$$
\frac{1}{k}{\mathbb E}\normo{A_{k}}
\leq
{\mathbb E}\normo{\mathcal{T}_{1}}
+\frac{p}{k}{\mathbb E}\left[\sumo_{i=1}^{k}\normo{\mathcal{T}_{i}}\right].
$$
From Lemma~\ref{L-EED}, the second term on the right-hand side may be replaced
by that of the form
$$
\frac{p}{k}\left({\mathbb E}\normo{\mathcal{T}_{1}}
+\frac{k-1}{\sqrt{2k-1}}\sqrt{{\mathbb D}\normo{\mathcal{T}_{1}}}\right),
$$
which tends to $0$ as $ k\rightarrow\infty $.
\end{proof}

\section{Discussion and Examples}\label{S-DAE}
Now we discuss the behaviour of the bounds (\ref{I-MCT}) under various
assumptions concerning the service times in the network. First note that the
derivation of the bounds does not require the $k$th service times
$ \tau_{ik} $ to be independent for all $ i=1,\ldots,n $. As it is easy 
to see, if $ \tau_{ik}=\tau_{k} $ for all $ i $, we have
$ \normo{{\mathbb E}[\mathcal{T}_{1}]}={\mathbb E}\normo{\mathcal{T}_{1}} $, and
so the lower and upper bound coincide.

To show how the bounds vary with strengthening the dependency, we consider the
network with $ n=5 $ nodes, depicted in Fig.~\ref{F-AFJ}. Let
$ \tau_{i1}=\sum_{j=1}^{5}a_{ij}\xi_{j1} $, where $ \xi_{j1} $,
$ j=1,\ldots,5 $, are i.i.d. random variables with the exponential
distribution of mean $ 1 $, and
$$
a_{ij}
=\left\{
   \begin{array}{ll}
    a,                & \mbox{if $ i=j$}, \\
    \frac{1}{4}(1-a), & \mbox{if $ i\ne j$},
   \end{array}
 \right.
$$
where $ a $ is a number such that $ 1\leq a\leq 1/5 $.

It is easy to see that for $ a=1 $, one has $ \tau_{i1}=\xi_{i1} $, and
therefore, $ \tau_{i1} $, $ i=1,\ldots,5 $, present independent random
variables. As $ a $ decreases, the service times $ \tau_{i1} $ become
dependent, and with $ a=1/5 $, we will have
$ \tau_{i1}=(\xi_{11}+\cdots+\xi_{51})/5 $ for all $ i=1,\ldots,5 $.

The next table presents estimates of the mean cycle time
$ \widehat{\gamma} $ obtained via simulation after performing 100000
service cycles, together with the corresponding lower and upper bounds
calculated from (\ref{I-MCT}).
\begin{table}[hhh]
\begin{center}
\begin{tabular}{||c|c|c|c||}
\hline\hline
\strut & & & \\
$\; a \;$ &
$\normo{{\mathbb E}[\mathcal{T}_{1}]}$ &
$\widehat{\gamma}$ &
${\mathbb E}\normo{\mathcal{T}_{1}}$ \\
\strut & & & \\
\hline
1   & 1.0 & 1.005718 & 2.283333 \\
1/2 & 1.0 & 1.002080 & 1.481250 \\
1/3 & 1.0 & 1.000871 & 1.213889 \\
1/4 & 1.0 & 1.000279 & 1.080208 \\
1/5 & 1.0 & 1.000000 & 1.000000 \\
\hline\hline
\end{tabular}
\caption{Numerical results for dependent service times.}
\end{center}
\end{table}

Let us now consider the network in Fig.~\ref{F-AFJ} under the assumption that
the service times $ \tau_{i1} $ are independent exponentially distributed
random variables. We suppose that $ {\mathbb E}[\tau_{i1}]=1 $ for all
$ i $ except for one, say $ i=4 $, with $ {\mathbb E}[\tau_{41}] $
essentially greater than $ 1 $. One can see that the difference between the
upper and lower bounds will decrease as the value of
$ {\mathbb E}[\tau_{41}] $ increases. The next table shows how the bounds
vary with different values of $ {\mathbb E}[\tau_{41}] $.
\begin{table}[hhh]
\begin{center}
\begin{tabular}{||c|c|c|c||}
\hline\hline
\strut & & & \\
${\mathbb E}[\tau_{41}]$ &
$\normo{{\mathbb E}[\mathcal{T}_{1}]}$ &
$\widehat{\gamma}$ &
${\mathbb E}\normo{\mathcal{T}_{1}}$ \\
\strut & & & \\
\hline
 1.0 &  1.0 &  1.005718 &  2.283333 \\
 2.0 &  2.0 &  2.004857 &  2.896032 \\
 3.0 &  3.0 &  3.004242 &  3.685531 \\
 4.0 &  4.0 &  4.003627 &  4.554525 \\
 5.0 &  5.0 &  5.003013 &  5.465368 \\
 6.0 &  6.0 &  6.002398 &  6.400835 \\
 7.0 &  7.0 &  7.001783 &  7.351985 \\
 8.0 &  8.0 &  8.001168 &  8.313731 \\
 9.0 &  9.0 &  9.000553 &  9.282968 \\
10.0 & 10.0 & 10.000008 & 10.257692 \\
\hline\hline
\end{tabular}
\caption{Numerical results for a fork-join network.}
\end{center}
\end{table}

Finally, let us discuss the effect of decreasing the variance
$ {\mathbb D}[\tau_{i1}] $ on the bounds on $ \gamma $. Note that if
$ \tau_{i1} $ were degenerate random variables with zero variance, the
lower and upper bounds in (\ref{I-MCT}) would coincide. One can therefore
expect that with decreasing the variance of $ \tau_{i1} $, the accuracy of
the bounds increases.

As an illustration, consider a tandem queueing system (see Fig.~\ref{F-OTQ})
with $ n=5 $ nodes. Suppose that $ \tau_{i1}=\xi_{i1}/r $, where
$ \xi_{i1} $, $ i=1,\ldots,5 $, are i.i.d. random variables which have the
Erlang distribution with the probability density function
$$
f_{r}(t)
=\left\{
   \begin{array}{ll}
    t^{r-1}e^{-t}/(r-1)!, & \mbox{if $ t>0$}, \\
    0,                    & \mbox{if $ t\leq 0$}.
   \end{array}
 \right.
$$
Clearly, $ {\mathbb E}[\tau_{i1}]=1 $ and
$ {\mathbb D}[\tau_{i1}]=1/r $. Related numerical results including
estimates $ \widehat{\gamma} $ evaluated by simulating 100000 cycles, and
the bounds on $ \gamma $, are shown below.
\begin{table}[hhh]
\begin{center}
\begin{tabular}{||c|c|c|c||}
\hline\hline
\strut & & & \\
$\; r \;$ &
$\normo{{\mathbb E}[\mathcal{T}_{1}]}$ &
$\widehat{\gamma}$ &
${\mathbb E}\normo{\mathcal{T}_{1}}$ \\
\strut & & & \\
\hline
 1 & 1.0 & 1.042476 & 2.928968 \\
 2 & 1.0 & 1.026260 & 2.311479 \\
 3 & 1.0 & 1.019503 & 2.045538 \\
 4 & 1.0 & 1.015637 & 1.890824 \\
 5 & 1.0 & 1.013110 & 1.787242 \\
 6 & 1.0 & 1.010864 & 1.711943 \\
 7 & 1.0 & 1.009920 & 1.654154 \\
 8 & 1.0 & 1.008409 & 1.608064 \\
 9 & 1.0 & 1.007726 & 1.570232 \\
10 & 1.0 & 1.006657 & 1.538479 \\
\hline\hline
\end{tabular}
\caption{Numerical results for tandem queues.}
\end{center}
\end{table}

\subsection*{Acknowledgments}
The author would like to thank WODES'98 reviewers for their helpful comments
and suggestions.

\bibliographystyle{utphys}

\bibliography{Bounds_on_mean_cycle_time_in_acyclic_fork-join_queueing_networks}

\end{document}